\documentclass[12pt]{article}
\usepackage{amsfonts}

\usepackage{graphicx}
\usepackage{amsmath}

\newtheorem{theorem}{Theorem}

\newtheorem{corollary}[theorem]{Corollary}

\newtheorem{definition}[theorem]{Definition}

\newtheorem{lemma}[theorem]{Lemma}

\newtheorem{proposition}[theorem]{Proposition}
\newtheorem{remark}[theorem]{Remark}

\newenvironment{proof}[1][Proof]{\textbf{#1.} }{\ \rule{0.5em}{0.5em}}

\begin{document}

\title{{\LARGE The killed Brox diffusion.
\footnote{This research was partially supported by the CONACYT and
the Universidad de Costa Rica.}}}
\author{
Jonathan Gutierrez-Pav\'{o}n
\thanks{Departamento de Matematicas, CINVESTAV-IPN, A. Postal 14-740,
Mexico D.F. 07000, MEXICO. Email: jjgutierrez@math.cinvestav.mx} and
Carlos G. Pacheco
\thanks{Email: cpacheco@math.cinvestav.mx}}

\date{}
\maketitle

\begin{abstract}
We carry out an study of the Brox diffusion with killing. It turns
out that when leaving fixed the environment one is able to recast
some theory of diffusion and differential operators to deal with the
ill-posed generator of the Brox diffusion. Our first main result is
to give a close form of the Green operator associated to the
generator, i.e. the inverse of the generator. We do so by setting
the Lagrange identity in this context. Then, we give explicit
expressions in quenched form of the probability density function of
the process; such object is given in terms of the spectral
decomposition using the eigenvalues and eigenfuntions of the
infinitesimal generator of the diffusion. Moreover, we characterize
the eigenvalues and eigenfuntions using some parsimonious stochastic
differential equations. This program is carried out using the theory
of Sturm-Liouville, which in fact we have adapted to deal with the
ill-posed random operators.
\end{abstract}

\\



\section{Introduction}

Probabilistic models with random environment are important models to capture some kind of external randomness that there exists in the medium. 
From this type of modelling different mathematical tools have been
deviced to analyze their complexities; probably a good starting
point to initiate being acquainted with the random-environment
phenomena is the book by B.D. Hughes \cite{Hughes}.

A model that has been consider to a study a particle moving in a
potential over a finite set is given by the following random
operator, see e.g. \cite{Halperin, Fuku},
\begin{equation*}
[Lf](x)=\frac{1}{2}\left\{f^{\prime\prime}(x)-V(x)f(x) \right\}, \
x\in [0,1] ,
\end{equation*}
where $V$ represent a white-noise potential. In Halperin
\cite{Halperin} it is called the Schr\"{o}dinger particle.

On the other hand, a popular model with random medium is what some
people call the Sinai's random walk, or also called the Temkin's
model, which run in discrete time and discrete state space. The
continuous time-space idealization of the Temkin's model is what we
call the Brox diffusion, precisely because TH. Brox carefully
studied it in \cite{Brox} in 1989, and proved the same asymptotic
behaviour as Y.G. Sinai did for the Temkin's model in 1982. In
\cite{Schumacher}, this model was studied as a diffusion with random
coefficients.

Different studies have been carried out to have more understanding
of the Brox diffusion, let us give a brief account. For instance in
\cite{Kawazu92} it is analyzed the limit behaviour of the process as
time evolves. Asymptotic behaviour regarding the first passage time
have been also studied in \cite{Kawazu97}. In \cite{Shi} it is
studied some asymptotics of the local time, and in \cite{Hu} more
ideas on sample path asymptotic are carried out. Interesting
formulas were discovered in \cite{Cheliotis05} in connection with
functional of the environment. Yet, without conditioning on the
environment more asymptotic information about the process was found
in \cite{Talet} when there is a drift in the environment. In
\cite{Cheliotis08} one finds some understanding on the paths
behaviour. In \cite{Kasahara}, limit behaviour about occupation time
is worked out. A relevant detail analysis on asymptotic dynamics of
the local time is done in \cite{Andreoletti}, and more recently in
\cite{HuLe} it is proposed stochastic differential equations driven
by the Brox process. Also, in \cite{Flandoli} one can find some
similar analysis for generators than our, without proposing explicit
invertible operators. We also mention an interesting study in
\cite{Bovier} of the Sinai's walk truncated in a finite interval,
where the authors also carry out an analysis of the eigenvalues and
the relation with the local minimum in the potential.

Strictly speaking the Brox diffusion becomes a diffusion only after
freezing an environment (i.e. conditioning on one realization of the
environment), which gives rise to the so-called quenched case, and
without conditioning it is called the annealed case. Informally
speaking, the generator $L$ acting on $f$ has the form
\begin{equation*}
[Lf](x)=\frac{1}{2}\left\{f^{\prime\prime}(x)-W^{\prime}(x)f^{\prime}(x)
\right\},
\end{equation*}
where $W$ represents the Brownian motion. Since in the quenched case
one is dealing with a bona fide diffusion, the whole apparatus of
diffusions, for instance the one manufactured by K. It\^{o} and H.P.
McKean \cite{ItoMcKean}, can be used for this model. Moreover, one
knows that behind the theory of diffusions it is lurking the
Sturm-Liouville theory of second order linear operators.

Although the generator of the Brox diffussion does not have
differentiable coefficients as the classical ones,
one can adapt many of the results in second order operators, which
helps to say a lot about the generator, and ultimately on the
stochastic process. It turns out that one can deal with the
generator using the inner product to define it rigourously, this
kind idea have been use before, say in \cite{Mathieu}, to define the
so-called Dirichlet forms associated to the generator.

Due to the importance of analyzing models with random medium in a
finite state space, we propose in this article to build the killed
version of the Brox diffusion, which ends up having a finite state
space. This version of the Brox process allows us to recast results
from the theory of Sturm-Liouville and analyze the generator in such
a way that it becomes feasible to write down an spectral
representation of the probability density function. Such
respresentation, as one might expect, is in terms of the
eigenfunction and eigenvalues of the generator. It should be mention
that an important tool in this analysis is the so-called Green
operator, that is the inverse of the generator. We are able to find
explicitly the Green operator, which become of tremendous help at
the time to analyze the eigenfunctions of the generator. At this
stage we have concentrated to the quenched case, leaving for further
investigation the annealed case.

Let us explain how this work is organized. In the coming section we
present the original construction of the Brox diffusion, making
emphasis on the domain of the generator, which we need to know for
our purposes. In section 3 we build the killed Brox process and give
results on its generator, in particular we make use of what people
call the Wronskian to analyze solutions of certain related equation.
In Section 4 we find the Green operator; additionally we need to
stablish the validity of the Lagrange identity in this context.
Section 5 has the spectral representation of the density function,
which is in terms of the eigenvalues and eigenfunctions. Precisely,
Section 6 has a dissection of the eigenfunctions and eigenvalues,
providing a way to obtain these elements. In doing so, we establish
a version of an oscillation theorem suited for our generator. At the
end of Section 6, we also use the Ricatti transformation to provide
yet another stochastic differential equation which is of interest.
There is one appendix at the end with the proof using the classic
Pr\"{u}fer method to study zeros of eigenfunctions, again, such
proof suited for our generator.

\section{Brox diffusion}

The Brox diffusion is usually described with the following
stochastic differential equation
\begin{equation} \label{ecuacion informal de brox}
dX_{t}=-\frac{1}{2}W'(X_{t})dt+dB_{t},
\end{equation}
where $B:= \{B_{t}: t\geq 0  \}$ is the standard Brownian motion,
and $W:= \{ W(x): x \in \mathbb{R} \}$ is a two sided Brownian
motion, and they both are independent from each other. Here $W'$
denotes the derivative of $W$, sometimes called the white noise.

The expression (\ref{ecuacion informal de brox}) needs to have a
rigorous meaning. It happens that one can use the associated
generator in order to construct properly the process. When looking
the equation (\ref{ecuacion informal de brox}) one could say that
the process $X:= \{ X_{t}:t\geq 0 \}$ has associated the
infinitesimal operator given by
\begin{equation*} \label{brox operator informal}
Lf(x):=\displaystyle\frac{1}{2} e^{W(x)}\frac{d}{dx} \left(
e^{-W(x)}\frac{df(x)}{dx}\right).
\end{equation*}
This corresponds to considering the scale function
\begin{equation}\label{brox funcion escala}
s(x):=\displaystyle\int_{0}^{x}e^{W(y)}dy,
\end{equation}
and the speed measure
\begin{equation}\label{broz medida velocidad}
m(A):=\displaystyle\int_{A} 2e^{-W(y)}dy, \; \; \mbox{for Borel
sets}\; A \subseteq \mathbb{R}.
\end{equation}

Then one considers rigourously the operator
$$\displaystyle \frac{d}{dm}\frac{d}{ds}f.$$
To learn about the derivatives with respect functions or measures we
refer the reader to \cite{Revuz, ItoMcKean, GP}. We want to see that
the process $X$ associated with this operator is indeed a diffusion,
when leaving fixed $W$. One can appeal to the theory in
\cite{ItoMcKean} to build a process $Y$ in natural scale through the
speed measure $m_{Y}$ and the local time. The procedure is done in
the following way:
\begin{equation*}
Y_{t}=B_{T^{-1}_{t}},
\end{equation*}
where
\begin{equation*}\label{T_{t}}
T_{t}:= \displaystyle
\frac{1}{2}\int_{-\infty}^{\infty}L_{t}(x)m_{Y}(dx).
\end{equation*}

The process $Y_{t}:=s(X_{t})$ is in natural scale. Now, we consider
the problem $3.18$, p.$310$ from \cite{Revuz}, and we obtain that
the new scale function and speed measure of the process $Y:=s(X)$
are
$$s_{Y}(x)=x,\text{ and }$$
$$m_Y(dx)=2e^{-2W(s^{-1}(x))}dx.$$

Then we define the Brox process as
$$X_{t}=s^{-1}(B_{T^{-1}_{t}}),$$
where
$$T_{t}:=\displaystyle\int_{0}^{t}e^{-2W(s^{-1}(B_{u}))}du.$$

Let us now describe the domain of the generator.
\begin{theorem}
For any environment $W$, the domain $D(L)$ is contained in the space
of differential functions $C^{1}(\mathbb{R}).$ Moreover the
derivative of a function $f$ in $D(L)$ takes the form
$f'(x)=e^{W(x)}g(x)$ with $g \in C^{1}(\mathbb{R}).$
\end{theorem}
\begin{proof}
According to Mandl \cite{Mandl}, p.$22$, if $h(x):=Lf(x)$ for $f \in
D(L)$, then
\begin{equation}\label{1}
f(x)=\displaystyle\int_{a}^{x}
\int_{a}^{y}h(z)dm(z)ds(z)+f(a)+\left[s(x)-s(a) \right]
\frac{df}{ds}(a),
\end{equation}
where
\begin{equation}\label{2}
\displaystyle\frac{df}{ds}(a)= \displaystyle \lim_{h \rightarrow
0}\frac{f(a+h)-f(a)}{s(a+h)-s(a)}=
\displaystyle\frac{f'(a)}{e^{W(a)}}.
\end{equation}

Thus, using (\ref{1}) and (\ref{2}), as well as the speed measure
(\ref{broz medida velocidad}) and the scale function (\ref{brox
funcion escala}), we have that
\begin{equation} \label{mandl f}
f(x)=2\displaystyle\int_{a}^{x}
\int_{a}^{y}h(z)e^{-W(z)}e^{W(y)}dzdy+f(a)+\displaystyle\frac{f'(a)}{e^{W(a)}}
\int_{a}^{x}e^{W(z)}dz.
\end{equation}

Therefore we obtain that $f \subseteq C^{1}(\mathbb{R}).$ Now, if we
calculate explicitly the derivative of $f$, we arrive to
\begin{equation} \label{mandl derivada de f}
f'(x)= e^{W(x)}\cdot \left[
2\displaystyle\int_{a}^{x}h(z)e^{-W(z)}dz +
\displaystyle\frac{f'(a)}{e^{W(a)}} \right].
\end{equation}

Then we have that $f'(x)=e^{W(x)} g(x) $ with $g\in
C^{1}(\mathbb{R}).$
\end{proof}

We use the previous construction of the Brox process for construct
the Brox process with killing in a and b.

\section{Brox process with killing}

Let $a<b$. We first consider the Brownian motion with killing in
$s(a)$ and $s(b)$, where $s$ is the scale function of the Brox
process, i.e.
$$s(x):=\displaystyle\int_{0}^{x}e^{W(y)}dy.$$
Since $s$ is non-drecreasing, $s(a)<s(b)$.

Hence, according to \cite{Borodin}, p.$105$, the infinitesimal
operator associated to the Brownian motion with killing at $s(a)$
and $s(b)$ is
$$L_{B}f(x)=\displaystyle\frac{1}{2}f''(x), \; \; s(a)<x<s(b),$$
where the domain of this operator is
$$D(L_{B})= \{f: L_{B}f \in C_{b}([s(a),s(b)]), f(s(a))=f(s(b))=0   \}.$$

We use this Brownian with killing and the construction in the
previous section to construct the Brox process with killing. We
consider the following idea, taken from \cite{Brox}, p.$1216$.

Start with the Brownian motion with killing in $s(a)$ and $s(b)$,
denote it by $\bar{B}$. Now, we consider
$$\bar{T}_{t}:=\displaystyle\int_{0}^{t}e^{-2W(s^{-1}(\bar{B}_{u}))}du,$$

Then the new process $\bar{X}$ defined as
$$\bar{X_{t}}=s^{-1}(\bar{B}_{\bar{\gamma}_{t}}),$$
is the Brox process with killing in $a$ and $b$. Where $s$ is the
scale function associated to the Brox process, and $\bar{\gamma}$ is
the inverse of $\bar{T}$.

Thus, the infinitesimal operator associated to $\bar{X}$ is
$$\bar{L}f(x)=\displaystyle\frac{d}{dm}\frac{d}{ds}f(x),$$
where $m$ and $s$ are the speed measure and the scale function of
the Brox process, and from general theory of diffusions, see e.g.
\cite{ItoMcKean, Borodin}, we know that the domain of $\bar{L}$ is
the set of functions in the domain of the generator of the Brox
process which are zero at $a$ and $b$. Therefore, the domain
$D(\bar{L})$ satisfies the following conditions:
\begin{itemize}
    \item $f(a)=0,$
    \item $f(b)=0,$
    \item $f'(x)= e^{W(x)} g(x), \ g\in C^{1}(\mathbb{R}).$
\end{itemize}

In general, $\bar{L}f$ is well defined whenever
$e^{-W(x)}f^{\prime}(x)$ is still differentiable, we denote this as
\begin{equation*}
D_{0}=\left\{ f\ \text{ differentiable}: e^{-W(x)}f^{\prime}(x)
\text{ is differentiable}  \right\}.
\end{equation*}
Now we present a couple of results regarding the generator and
solutions of the so-called eigenfunction equation. These results
will be very useful for the rest of the paper. In the following
result we regard $\bar{L}$ as an operator acting not just on
$D(\bar{L})$ but on any function in $D_{0}$. Abusing the notation,
sometimes we write $f(x,\lambda)$ in liue of $f(x)$ to emphazise the
dependence with the spectral parameter $\lambda$.

\begin{proposition}\label{PropPsiEq}
i) The operator $\bar{L}$
can be applied as
$$\bar{L}f(x)= \displaystyle\frac{e^{W(x)}}{2}\left(e^{-W(x)}f'(x) \right)',$$
for any differentiable functions $f\in D_{0}$. 
\\ ii) For any $\lambda>0$, the problem
$$\bar{L}g+\lambda g=0, \text{ given } g(a,\lambda)=0 \text{ and }
g^{\prime}(a,\lambda)=1,$$ admits a solution in $D_{0}$, where $g^{\prime}$ means the derivative with respect to $x$.
Moreover, such solution satisfies the equation
\begin{equation}\label{ecuacion integral de g}
g(x,\lambda)=-2\lambda
\displaystyle\int_{a}^{x}\int_{a}^{y}g(z,\lambda)e^{-W(z)}e^{W(y)}dzdy
+ \displaystyle\frac{1}{e^{W(a)}} \int_{a}^{x}e^{W(z)}dz.
\end{equation}
iii) The reciprocal of ii) is also true. That is, if $g$ satisfies
(\ref{ecuacion integral de g}), then $g$ solves the problem
$\bar{L}g+ \lambda g=0$, $g(a,\lambda)=0$ and
$g^{\prime}(a,\lambda)=1$.
\end{proposition}
\begin{proof}
i) We simply calculate $\displaystyle\frac{d}{dm}\frac{d}{ds}f(x)$
using the speed measure $m$ and the scale function $s$. Then we have
\begin{eqnarray*}
  \displaystyle\frac{d}{d m}\frac{d}{d s}f(x) &=&
  \frac{d}{d m} \left( \displaystyle\lim_{h \rightarrow 0}
  \frac{f(x+h)- f(x)}{\int_{x}^{x+h} e^{W(y)}dy}  \right) \\
   &=& \frac{d}{d m} \left( e^{-W(x)}f'(x)  \right) \\
   &=& \displaystyle\lim_{h \rightarrow 0}
   \frac{e^{-W(x+h)f'(x+h)}-e^{-W(x)f'(x)}}
   {\int_{x}^{x+h} 2e^{-W(y)}dy} \\
   &=& \displaystyle\frac{e^{W(x)}}{2} \left(
e^{-W(x)}f'(x) \right)'.
\end{eqnarray*}
ii) The existence of a solution comes from general theory of
diffusions, see e.g. \cite{ItoMcKean}, specifically in Section 4.6,
page 128. To prove the second statement, we apply twice a
fundamental theorem of calculus adapted to derivatives with respect
a measure, see e.g. \cite{Revuz, GP}. Indeed, using the definition
of $\bar{L}f(x)=\displaystyle\frac{d}{dm}\frac{d}{ds}f(x)$ in the
equation $\bar{L}g+\lambda g=0$, we calculate first the integral
with respect to $m$ and later with respect to
$s$, from where equation (\ref{ecuacion integral de g}) arises.\\
iii) The proof of this fact is simply to calculate on the equation
the derivative with respect to $s$ and later with respect a $m$.
\end{proof}


In the theory of differential equations, the the so-called Wronskian
is an important tool to detect whether there is dependance among
solutions of a differential equation. As one might expect, it turns
out that in this context with a random coefficients the Wronskian
can be calculated and it is used as well for the same purpose. It is
defined as the following determinant,
\begin{equation*}
w_{f,g}(x):=\left|
\begin{array}{cc}
g(x) & f(x)\\
\frac{dg}{ds}(x) & \frac{d f}{ds}(x)
\end{array}
 \right|
=\frac{f'(x)g(x)}{e^{W(x)}}-\frac{f(x)g'(x)}{e^{W(x)}}.
\end{equation*}
where $f$ and $g$ are two functions, and $e^{W(x)}$ is the density
of the scale function associated to the Brox process.

\begin{proposition}\label{PropWronskian}
Let $f$ and $g$ be two solutions of $\bar{L}\psi+ \lambda \psi=0$,
with $x\in [a,b]$, such that $f(a,\lambda)=0$ and $g(a,\lambda)=0$.
Then $\frac{d w_{f,g}(x)}{dx}=0$, which implies that there exists a
constant $C$ such that $f=Cg$.
\end{proposition}
\begin{proof}
Using Proposition \ref{PropPsiEq}, if $f$ and $g$ are solution of
$\bar{L}\psi+ \lambda \psi=0$, then
\begin{equation}\label{EasyL}
\frac{e^{W(x)}}{2}\left(e^{-W(x)}\psi^{\prime}(x)\right)^{\prime}+\lambda
\psi(x)=0,
\end{equation}
which implies
\begin{equation*}
\left( \frac{\psi^{\prime}(x)}{e^{W(x)}}
\right)^{\prime}=(-2)\frac{\lambda \psi(x)}{e^{W{x}}}.
\end{equation*}
Hence, we can substitute into the derivative of the Wronskian to see
that
\begin{eqnarray*}
  \frac{d w_{f,g}(x)}{dx} &=&
  \left( \frac{f'(x)g(x)}{e^{W(x)}}-\frac{f(x)g'(x)}{e^{W(x)}} \right)' \\
   &=& \left(\frac{f'(x)}{e^{W(x)}} \right)'g(x)-
  \left(\frac{g'(x)}{e^{W(x)}} \right)'f(x)\\
   &=& 0
\end{eqnarray*}
This implies that $w_{f,g}(x)=M$ for some constant $M$, so
\begin{equation*}
f'(x)g(x)-g'(x)f(x)=Me^{W(x)},
\end{equation*}
where $M$ is a constant. Using that $f(a,\lambda)=0$ and
$g(a,\lambda)=0$, we have that $M=0$, then for all $x\in [a,b]$ we
have that
\begin{equation} \label{animal}
w_{f,g}(x)= f'(x)g(x)-g'(x)f(x)=0
\end{equation}

Consider the set $A:=\{x: g(x)\neq 0 \}$ . For all $x \in A$ we have

$$\displaystyle\frac{f'(x)g(x)-g'(x)f(x)}{g^{2}(x)}=0,$$

\noindent this implies

$$\left( \displaystyle\frac{f(x)}{g(x)} \right)'=0.$$

Then we have that there exists a constant $C$ such that $f(x)=Cg(x)$
for all $x \in A$, which by (\ref{animal}) also holds for $x \in
A^{c}$. To see that, let $B:=\{x: g'(x)\neq 0 \}$. Then, using the
equation (\ref{animal}), we have for all $x \in B$

$$f(x)= \displaystyle\frac{f'(x)}{g'(x)}g(x)$$

We now show that $\displaystyle \frac{f'(x)}{g'(x)}$ is a constant.
To do that, we calculate the derivative. Using again the fact that
$f$ and $g$ are solutions of (\ref{EasyL}) and taking into account
(\ref{animal}),
we have
\begin{eqnarray*}
  \left( \displaystyle \frac{f'(x)}{g'(x)} \right)' &=&
  \left( \displaystyle \frac{f'(x)e^{-W(x)}}{g'(x)e^{-W(x)}} \right)'\\
   &=& \frac{g'(x)e^{-W(x)}\cdot (f'(x)e^{-W(x)})'-f'(x)e^{-W(x)}
  \cdot (g'(x)e^{-W(x) })'}{(g'(x)e^{-W(x)})^{2}} \\
   &=& \frac{g'(x)e^{-W(x)}\cdot (- 2\lambda f(x)e^{-W(x)})-f'(x)e^{-W(x)}
  \cdot (-2\lambda g(x)e^{-W(x) })}{(g'(x)e^{-W(x)})^{2}} \\
   &=& 0.
\end{eqnarray*}

Then for all $x \in B$ there exists a constant $K$ such that $f(x)=K
g(x)$.  We have that $C=K$. To do that,it is sufficient to show that
$g$ and $g'$ do not have common zeros. This is true by using the
formulae (\ref{mandl f}) and (\ref{mandl derivada de f}).

Therefore for all $x \in [a,b]$ we have
\begin{equation}\label{EqSWr0}
f(x)=C g(x).
\end{equation}

And the proof is done.
\end{proof}


To facilitate the notation, from now on we use $L$ alone instead of
$\bar{L}$.

\section{Green Operator}

From the previous section we have that the infinitesimal operator
associated with the Brox diffusion with killing on $a$ and $b$ is
given by
$$Lf(x)= \displaystyle\frac{e^{W(x)}}{2}\left(e^{-W(x)}f'(x)  \right)',$$
whose domain are functions $f\in D_{0}$ such that $f(a)=f(b)=0$,
i.e. $f \in D(L).$

We want to construct the so-called Green operator, which is actually
the inverse operator of $L$. First, consider the following identity
known as the Lagrange identity.
\begin{lemma} \label{identidad de lagrange}
Let $f,g$ in $D_{0}$, then
$$2 e^{-W(x)} \left[ g(x) Lf(x)- f(x) Lg(x) \right]=
\left[e^{-W(x)}\left(f'(x)g(x)-f(x)g'(x) \right)  \right]'.$$
\end{lemma}
\begin{proof}
Let $h_{1}(x)=e^{-W(x)}f^{\prime}(x)$ and
$h_{2}(x)=e^{-W(x)}g^{\prime}(x)$. Then
\begin{eqnarray*}
\left[e^{-W(x)}\left(f'(x)g(x)-f(x)g'(x) \right)  \right]'&=&
\left[h_{1}(x)g(x)-h_{2}(x)f(x)\right]'\\
&=&
h_{1}^{\prime}(x)g(x)+g^{\prime}(x)h_{1}(x)-h_{2}^{\prime}(x)f(x)-f^{\prime}(x)h_{2}(x)\\
&=& h_{1}^{\prime}(x)g(x)-h_{2}^{\prime}(x)f(x)\\
&=& 2 e^{-W(x)} \left[ g(x) Lf(x)- f(x) Lg(x) \right].
\end{eqnarray*}
\end{proof}

It is important to notice in this formula that it is not necessary
to differentiate $e^{W(x)}$.

From previous identity, after integrating we also obtain
\begin{corollary}\label{lagrange2}
Let $\alpha$, $\beta$ be such that $a\leq \alpha < \beta \leq b$.
And $\lambda_{1}$ and $\lambda_{2}$ with $Lf+\lambda_{1}f=0$,
$Lg+\lambda_{2}g=0$. Then
$$ \left[e^{-W(x)}\left(f'(x)g(x)-f(x)g'(x) \right) \right]_{\alpha}^{\beta}=
2(\lambda_{2}-\lambda_{1})\displaystyle\int_{\alpha}^{\beta}e^{-W(x)}f(x)g(x)dx.$$
\end{corollary}

Using methods developed in the companion paper \cite{GPP}, see also
\cite{Pacheco}, we construct the Green operator. Consider $f$ in the
domain of $L$. We propose the following operator.
\begin{definition}\label{kernelO}
Let
\begin{equation}\label{kernel}
\noindent g(x,\xi):= \displaystyle \left\{
\begin{array}{ll}
   -Cu(x)v(\xi) ,  & {a \leq x \leq \xi;} \\
   -Cu(\xi)v(x),  & {\xi \leq x \leq b,} \\
\end{array}
\right.
\end{equation}
where
$$C:= \displaystyle\int_{a}^{b}e^{W(z)}dz,\
u(x):= \displaystyle\frac{1}{C}\int_{a}^{x}e^{W(z)}dz,\ v(x):=
\displaystyle\frac{1}{C}\int_{x}^{b}e^{W(z)}dz.$$ We then define the
Green operator as
\begin{equation*}
Tf(x):=\displaystyle\int_{a}^{b}2e^{-W(z)}g(z,x)f(z)dz,
\end{equation*}
for any $f\in D(L)$.
\end{definition}

Let $\xi$ be fixed. Note that $g$ satisfies
$$g(a,\xi)=0, \; \; g(b,\xi)=0.$$
Let us also see that $Lg(x,\xi)=0$ for $x \neq \xi$, where $\xi$ is
any fixed value, and with the understanding that $g'(x,\xi)$ is the
derivative of $g(x,\xi)$ with respect to the first argument $x$,
with $x \neq \xi$. To do that, we consider the derivatives from
right and left of $\xi$. Then we obtain
\begin{center}
\noindent $g^{\prime}(x,\xi):= \displaystyle \left\{
\begin{array}{ll}
   -e^{W(x)}v(\xi) ,  & {a \leq x \leq \xi;} \\
   e^{W(x)}u(\xi),  & {\xi \leq x \leq b,} \\
\end{array}
\right.$
\end{center}
Then, after multiplying by $e^{-W(x)}$ there is no more dependance
on $x$, therefore the second derivative gives $0$. This implies that
$Lg(x,\xi)=0$ for $x \in [a,\xi)$ and $x \in (\xi,b]$.

Now we use the Lagrange identity with the function $f$ and $g$,
using that $Lg(x,\xi)=0$. Then for $x \in [a,\xi)$ and $x \in
(\xi,b]$ we have
\begin{equation}\label{lagrange con g}
2 e^{-W(x)} \left[ g(x,\xi) Lf(x) \right]=
\left[e^{-W(x)}\left(f'(x)g(x,\xi)-f(x)g'(x,\xi) \right)  \right]'.
\end{equation}

On integrating both sides of (\ref{lagrange con g}) on the intervals
$(a,\xi^{-})$ and $(\xi^{+},b)$, where $\xi^{-}:= \xi-\epsilon$ and
$\xi^{+}:=\xi+\epsilon$.
Then, to calculate $TLf$, we have the following two equalities
\begin{equation}\label{lagrange a,xi}
\displaystyle\int_{a}^{\xi^{-}}2e^{-W(x)}g(x,\xi)Lf(x)dx =
\left[e^{-W(x)} \left(f'(x)g(x,\xi)- f(x)g'(x,\xi) \right)
\right]_{a}^{\xi^{-}},
\end{equation}
\begin{equation}\label{lagrange xi,b}
\displaystyle\int_{\xi^{+}}^{b}2e^{-W(x)}g(x,\xi)Lf(x)dx =
\left[e^{-W(x)} \left(f'(x)g(x,\xi)- f(x)g'(x,\xi) \right)
\right]_{\xi^{+}}^{b}.
\end{equation}

By adding (\ref{lagrange a,xi}) and (\ref{lagrange xi,b}), and using
that $f(a)=f(b)=g(a,\xi)=g(b,\xi)=0$, we have
\begin{equation*}
\displaystyle\int_{a}^{b}2e^{-W(x)}g(x,\xi)Lf(x)dx-\displaystyle\int_{-\epsilon}^{\epsilon}2e^{-W(x)}g(x,\xi)Lf(x)dx
\end{equation*}
\begin{eqnarray*}
&=&e^{-W(\xi^{-})}\left(f^{\prime}(\xi^{-})g(\xi^{-},\xi)-f(\xi^{-})g^{\prime}(\xi^{-},\xi)\right)\\
&-&e^{-W(\xi^{+})}\left(f^{\prime}(\xi^{+})g(\xi^{+},\xi)-f(\xi^{+})g^{\prime}(\xi^{+},\xi)\right).
\end{eqnarray*}
After expanding we end up with four terms. From the continuity of
$W, \ f^{\prime}$ and $g$, the first and third terms cancel each
other when $\epsilon\to 0$.

Since $g^{\prime}$ is not continuous, the second and fourth terms do
not vanish. These terms are
\begin{equation*}
-e^{-W(\xi^{-})}f(\xi^{-})g^{\prime}(\xi^{-},\xi)
+e^{-W(\xi^{+})}f(\xi^{+})g^{\prime}(\xi^{+},\xi).
\end{equation*}
Taking the discountinuity into account, previous display is
\begin{equation*}
e^{-W(\xi^{-})}f(\xi^{-})e^{W(\xi^{-})}v(\xi)
+e^{-W(\xi^{+})}f(\xi^{+})e^{W(\xi^{+})}u(\xi).
\end{equation*}
Then, when $\epsilon\to 0$, it becomes
\begin{equation*}
e^{-W(\xi)}f(\xi)e^{W(\xi)}u(\xi)+e^{-W(\xi)}f(\xi)e^{W(\xi)}v(\xi).
\end{equation*}
Since $u(\xi)+v(\xi)=1$,
\begin{equation*} \label{inversa del operador}
TLf(\xi)= \displaystyle\int_{a}^{b}2e^{-W(x)}g(x,\xi)Lf(x)dx =
f(\xi).
\end{equation*}

The conclusion is given as follows.
\begin{theorem}\label{ThmLT}
Let $T$ given in Definition (\ref{kernelO}). Then  $T$ satisfies
$T(Lf)(x)= f(x)$ for all $f \in D(L)$, and $L(Th)(x)=h(x)$ for all
$h \in L^{2}([a,b])$.
\end{theorem}
\begin{proof}
We have already shown that $T(Lf)(x)=f(x)$. Using that
\begin{equation*}
Th(x)= -2Cv(x)\int_{a}^{x}e^{-W(z)}u(z)h(z)dz-2C
u(x)\int_{x}^{b}e^{-W(z)}v(z)h(z)dz,
\end{equation*}
One can apply $L$ to verify that $LTh=h$.
\end{proof}

\section{Toward density}

In the theory of Markov processes, it is well known that spectral
information of the generator helps to study the transition
probability functions of the stochastic process. In turn, one can
use the eigenvalues and eigenfunctions to give expressions for the
probability density. In fact, we can identify the eigenvalues of the
generator of the killed Brox process with the eigenvalues of the
Green operator $T$ of Theorem \ref{ThmLT}, precisely because $T$ is
the inverse of $L$.
\begin{corollary}
Operator $L$ has almost surely a countable set of eigenvalues.
\end{corollary}
\begin{proof}
This comes from the fact that for almost every trajectory of $W$,
the operator $T$ is a compact operator, thus it has a countable set
of eigenvalues. Then, if $(\lambda,f)$ is an eigenpair of $T$, then
$Tf+\lambda f=0$. Thus, $f=LTf=-\lambda Lf$, i.e. $(1/\lambda,f)$ is
an eigenpair of $L$.
\end{proof}

Now we know that the generator of $X$ has a discrete spectrum given
by the eigenvalues $\lambda_{n}$, and each one has associated an
eigenfunction $\phi_{n}$. Thus, at a theoretical point of view, it
is just a matter to join pieces to have the spectral decomposition
of the probability transition function.

Notice that apriori we do not know if the transition probabilities
are absolutely continuous with respect to the Lebesgue measure,
however it is indeed the case.
\begin{theorem}
If we leave fixed an environment $W$, then for all $x,y \in (a,b)$
we have
\begin{equation}\label{EqTDF}
p(t,x,y)=2e^{-W(y)}\displaystyle\sum_{n=1}^{\infty}
e^{-\lambda_{n}t}\phi_{n}(x)\phi_{n}(y),
\end{equation}
where $p(t,x,y)$ is the density function of $X_{t}$ given that
$X_{0}=x$, and $\left\{ \lambda_{n},\ \phi_{n}
\right\}_{n=1}^{\infty}$ are the eigenvalues and eigenfunctions of
$L$.
\end{theorem}

Some properties known in the classic case, where the parameters of
the operator $L$ are differentiable functions, are also known in the
case the same parameters are not necessarily differentiable.

\begin{proof}
The set $\{\phi_{n}\}_{n=1}^{\infty}$ forms a basis for the space
$L^{2}\left( [a,b],2e^{-W(x)}  \right)$, where the inner product is
given by $\langle f,g \rangle:= \displaystyle
\int_{a}^{b}f(x)g(x)2e^{-W(x)}dx,$ see e.g. Theorem 4.6.2 point (5)
in \cite{Zettl}.

The operator $L$ is self-adjoint and non-positive on $D(L)$.
Therefore if we leave fixed an environment $W$, the semigroup
$P_{t}f$ can be written as
$P_{t}(f(x))=\sum_{n=1}^{\infty}e^{-\lambda_{n}t}\langle f,
\phi_{n}\rangle \phi_{n}(x),$ see \cite{linetsky}.

On the other hand, for $W$ fixed, we have that the semigroup can be
written as $$P_{t}(f(x)):=E_{x}\left(f\left(X_{t} \right) \right),$$
where $X_{t}$ is the Brox Process with killing on $a$ and $b$.

Now, since $P_{t}(f(x)):=
E(f(X_{t})|X_{0}=x)=\displaystyle\int_{-\infty}^{\infty}f(y)p(t,x,dy)$,
using dominated convergence theorem we arrive at
\begin{eqnarray*} \displaystyle\int_{-\infty}^{\infty}f(y)p(t,x,dy)
 &=& \sum_{n=1}^{\infty}e^{-\lambda_{n}t}\langle f,\phi_{n} \rangle \phi_{n}(x) \\
 &=& \sum_{n=1}^{\infty}e^{-nt}
 \left(\displaystyle\int_{-\infty}^{\infty}f(y)\phi_{n}(y)   2e^{-W(y)}dy \right)
\phi_{n}(x) \\   &=& \displaystyle\int_{-\infty}^{\infty} f(y)
\left( \sum_{n=1}^{\infty}   e^{-\lambda_{n}t}   \phi_{n}(y)
\phi_{n}(x) 2e^{-W(y)} \right)dy .
\end{eqnarray*}
This proves the absolutely continuity and formula (\ref{EqTDF}).
\end{proof}

\section{Spectral analysis of the generator}

In previous section we have shown, at least at a theoretical level,
how one can give an spectral decomposition for the densities of $X$.
Let us go further to try to find or characterize the components of
such representation, that is to say: the eigenvalues and the
eigenfunctions. We deal first with the eigenfunctions and after with
the eigenvalues. We will keep noticing how the Green operator $T$ of
Theorem \ref{ThmLT} will be useful for our analysis.

\subsection{Eigenfunctions}


Let $\phi$ be an eigenfunction and $\lambda$ an eigenvalue of the
generator $L$, then it holds
$$L\phi+\lambda \phi=0\text{ with }\phi(a)=\phi(b)=0.$$

The Green operator gives the identity $TL\phi=\phi = -\lambda
T\phi$, that is
$$\phi(x)=\displaystyle -2 \lambda\int_{a}^{b} e^{-W(z)}\phi(z)g(z,x)dz.$$

From the definition of $g$ previous display becomes
$$\phi(x)= 2 C \lambda v(x) \displaystyle\int_{a}^{x}u(z)e^{-W(z)}\phi(z)dz
+ 2 C \lambda u(x) \displaystyle\int_{x}^{b}v(z)e^{-W(z)}\phi(z)dz.
$$ After taking the derivative, a cancellation occurs that yields
$$\phi'(x)= 2 \lambda e^{W(x)}
\left[\displaystyle\int_{x}^{b}v(z)e^{-W(z)}\phi(z)dz -
 \displaystyle\int_{a}^{x}u(z)e^{-W(z)}\phi(z)dz \right].$$

Consider the following function, a trick borrow from \cite{Karatzas}
page 269, which is previous display writing $t$ in lieu of $x$, and
$x$ in lieu of $W(x)$:
\begin{equation*} \label{funcion h}
h(t,x):=2 \lambda e^{x}
\left[\displaystyle\int_{t}^{b}v(z)e^{-W(z)}\phi(z)dz -
 \displaystyle\int_{a}^{t}u(z)e^{-W(z)}\phi(z)dz \right].
\end{equation*}

Applying the It$\hat{\mbox{o}}$'s formula to function $h$ one has
\begin{eqnarray*}
h(t, W(t))-h(a, W(a))&=& 2\lambda \int_{a}^{t}e^{W(s)}\left(
-v(s)e^{-W(s)}\phi(s)-u(s)e^{-W(s)}\phi(s)\right)ds\\
&+&\frac{1}{2}\int_{a}^{t}h(s,W(s))ds+\int_{a}^{t}h(s,W(s))dW(s).
\end{eqnarray*}

Since $\phi^{\prime}(t)=h(t,W(t))$, using $u(s)+v(s)=1$, we have
that $\phi$ satisfies the following stochastic differential
equation.
\begin{proposition}
Let $\phi$ be an eigenfunction of $L$ associated to the eigenvalue
$\lambda$. Then, $\phi$ is solution of
\begin{equation*}
d\phi'(t)= \left[-2 \lambda \phi(t) +\frac{1}{2} \phi'(t) \right]dt
+ \phi'(t)dW(t),
\end{equation*}
with conditions $\phi(a)= 0$ and $\phi(b)=0$.
\end{proposition}

\subsection{Eigenvalues}

In this section we give a method to deal with the eigenvalues. To do
that, with the aid of the Sturm-Liouville theory, we develop few
results suited to work with our operator $L$.

\begin{theorem}\label{teorema de oscilacion}
Consider functions $f\in D_{0}$. Define the following two operators
for $a<x<b$,
\begin{equation*}
L_{1}f(x):=\left(e^{-W(x)}f'(x) \right)'+ 2\lambda_{1}e^{-W(x)}f(x),
\;
\end{equation*}
\begin{equation*}
L_{2}f(x):=\left(e^{-W(x)}f'(x) \right)'+ 2\lambda_{2}
e^{-W(x)}(x)f(x), \;
\end{equation*}
where $\lambda_{2}> \lambda_{1}$. Let $\phi_{1}$ and $\phi_{2}$ such
that $L_{1}\phi_{1}=L_{2}\phi_{2}=0$. Then, between two zeros of
$\phi_{1}$ there is a zero of $\phi_{2}$. Moreover, if
$\phi_{1}(a)=\phi_{2}(a)=0$, then $\phi_{2}$ has at least as many
zeros as $\phi_{1}$ on $[a,b]$.
\end{theorem}
\begin{proof}
Suppose that $x_{1}$ and $x_{2}$ are two successive zeros of
$\phi_{1}$, and that $\phi_{2}(x)\neq 0$ for any $x\in
(x_{1},x_{2})$. Without loss of generality we assume that
$\phi_{1}(x)>0$ and $\phi_{2}(x)>0$ for any $x\in (x_{1},x_{2})$.

The Lagrange's identity in Lemma \ref{lagrange2} gives
$$\left[e^{-W(x)}\left(\phi_{1}'(x)\phi_{2}(x)-
\phi_{1}(x)\phi_{2}'(x)  \right)  \right]_{x_{1}}^{x_{2}}=
2(\lambda_{2}-\lambda_{1})\displaystyle\int_{x_{1}}^{x_{2}}
e^{-W(x)}\phi_{1}(x)\phi_{2}(x)dx.$$ Note that the right hand side
is strictly positive. However, the left hand side reduces to
$$e^{-W(x_{2})}\phi_{1}'(x_{2})\phi_{2}(x_{2})-
e^{-W(x_{1})}\phi_{1}'(x_{1})\phi_{2}(x_{1}).$$ Using the
assumptions of $\phi$ on $x_{1}$ and $x_{2}$, we observe that
$\phi_{2}(x_{2})\geq 0$, $\phi_{1}'(x_{2}) \leq0$,
$\phi_{2}(x_{1})\geq 0$ and $\phi_{1}'(x_{1}) \geq 0$, then the
above expression is less or equal to $0$, giving a contradiction.
Therefore $\phi_{2}$ has a zero between $x_{1}$ and $x_{2}$.

In particular, if $\phi_{1}(a)=\phi_{2}(a)=0$ and
$\phi_{1}(x_{1})=0$ with $a<x_{1}<b$, then there exists $z$, with
$a<z<x_{1}$ such that $\phi_{2}(z)=0$. Thus $\phi_{2}$ has at least
as many zeros as $\phi_{1}$ on $[a,b]$.
\end{proof}

\begin{corollary} \label{CoroZerosEigen}
If $\phi_{n}$ is an eigenfunction of $L$ associated with
$\lambda_{n}$, with $n=1,2...$, then $\phi_{n}$ has exactly $n-1$
zeros in the interval $(a,b)$.
\end{corollary}
We left the proof in the Appendix. The classic proof of this result
uses the so-called method of Pr\"{u}fer which is based in a change
of coordinates. The original proof for the standard equation is
difficult to find in the literature, one can find it thought in
\cite{Coddington}, from where we adapted it to our situation.


Using previous two results we can to show the following theorem.

\begin{theorem} \label{teorema funcion psi y sus ceros}
Let $\lambda \in \mathbb{R}$ be fixed, and let $\psi(x,\lambda)$ be
solution of $L\psi(x,\lambda)+ \lambda \psi(x,\lambda)=0$, $x\in
(a,b)$, that satisfies $\psi(a,\lambda)=0$ and $\psi'(a,\lambda)=1$.
Then the number of zeros of the map $x\mapsto \psi(x,\lambda)$ on
$(a,b]$ equals the number of eigenvalues of $L$ less or equal to
$\lambda$.
\end{theorem}

\begin{proof}
First, from Propostion \ref{PropPsiEq}, we know that such function
$\psi$ really exists.

The proof relies on Theorem \ref{teorema de oscilacion} and
Corollary \ref{CoroZerosEigen}. In what follows $\phi_{n}$ is the
eigenfunction associated with the $n$-eigenvalue $\lambda_{n}$, i.e.
$L\phi_{n}+\lambda_{n} \phi_{n}=0$, with
$\phi_{n}(a)=\phi_{n}(b)=0$.

Fix $\lambda$. We first suppose that there exist only $n$
eigenvalues less or equal to $\lambda$, i.e. $\lambda_{1}< ...<
\lambda_{n} \leq \lambda < \lambda_{n+1}$, and let us prove that the
map $x\mapsto \phi(x, \lambda)$ has exactly $n$ zeros in $(a,b]$. By
the Corollary \ref{CoroZerosEigen}, $\phi_{n}$ has exactly $n-1$
zeros on the open interval $(a,b)$, thus it has $n+1$ zeros in
$[a,b]$. Since $\lambda_{n} \leq \lambda$, by Theorem \ref{teorema
de oscilacion} we know that between two consecutive zeros of
$\phi_n$ there is one zero for $\psi$. Then $\psi$ has at least $n$
zeros on $(a,b]$, i.e. it has at least $n+1$ zeros on $[a,b]$.
However, if $\psi$ had $n+1$ zeros on $(a,b]$, again using Theorem
\ref{teorema de oscilacion} with $\lambda<\lambda_{n+1}$, the $n+2$
zeros of $\psi$ on $[a,b]$ would imply that $\phi_{n+1}$ had $n+1$
zeros on $(a,b)$, which is not the case. We conclude that $\psi$ has
exactly $n$ zeros on $(a,b]$.

On the other hand, we now suppose that $\psi$ has exactly $n$ zeros
in $(a,b]$. Let us now show that there exist only $n$ eigenvalues
less or equal that $\lambda$. Suppose by contradiction that the
eigenvalue $n+1$ is less or equal to $\lambda$, i.e. $\lambda_{n+1}
\leq \lambda$. If $\lambda = \lambda_{n+1}$ we have that $\psi =
\phi_{n+1}$, and by the Corollary \ref{CoroZerosEigen} $\psi$ has
$n$ zeros in $(a,b)$, and since $\phi_{n+1}(b)=0$, we obtain that
$\psi$ has $n+1$ zeros in $(a,b]$, which is a contradiction. If
$\lambda_{n+1} < \lambda$, by the Corollary \ref{CoroZerosEigen} we
know that $\phi_{n+1}$ has $n+2$ zeros in $[a,b]$. Now, by Theorem
\ref{teorema de oscilacion} we have that $\psi$ should have at least
$n+2$ in $[a,b]$, this implies that $\psi$ has $n+1$ zeros or more
in $(a,b]$, which is again a contradiction.

We know now that if $\psi$ has $n$ zeros in $(a,b]$, the $n+1$
eigenvalue satisfies  $\lambda < \lambda_{n+1}$. We will now show
that $\lambda_{n} \leq \lambda$, i.e. there exist only $n$
eigenvalues less or equal to $\lambda$.

Suppose that $\lambda < \lambda_{n}$. Recall that we are supposing
that $\psi$ has $n$ zeros in $(a,b]$, since $\psi(a)=0$, it has
$n+1$ zeros in $[a,b]$. Again, appealing to the Theorem \ref{teorema
de oscilacion}, since $\phi_{n}(a)=\phi_{n}(b)=0$, we have that
$\phi_{n}$ has at least $n+2$ zeros in $[a,b]$. We are saying that
$\phi_{n}$ has $n$ zeros or more in $(a,b)$, which contradicts
Corollary \ref{CoroZerosEigen}. And the proof is completed.
\end{proof}

\begin{remark}
Let us give a characterization of function $\psi$ of previous
theorem, i.e. $\psi$ such that
$$L\psi(x,\lambda)+ \lambda \psi(x,\lambda)=0,\ x\in (a,b)$$
with $\psi(a,\lambda)=0$ and $\psi'(a,\lambda)=1$.

From ii) of Proposition \ref{PropPsiEq}, $\psi$ is solution of the
equation
$$\psi(x,\lambda)=-2\lambda \displaystyle\int_{a}^{x}
\int_{a}^{y}\psi(z,\lambda)e^{-W(z)}e^{W(y)}dzdy+\psi(a,\lambda)+
\displaystyle\frac{\psi'(a,\lambda)}{e^{W(a)}}
\int_{a}^{x}e^{W(z)}dz.$$

By differentiating and taking into account the initial conditions,
we have the following two equations,
\begin{equation}\label{EqPsi}
\psi(x,\lambda)=-2\lambda \displaystyle\int_{a}^{x}
\int_{a}^{y}\psi(z,\lambda)e^{-W(z)}e^{W(y)}dzdy+
\displaystyle\frac{1}{e^{W(a)}} \int_{a}^{x}e^{W(z)}dz,
\end{equation}
\begin{equation*}
\psi'(x,\lambda)= e^{W(x)}\cdot \left[
-2\lambda\displaystyle\int_{a}^{x}\psi(z,\lambda)e^{-W(z)}dz +
\displaystyle\frac{1}{e^{W(a)}} \right].
\end{equation*}

\end{remark}


We finally arrive to the point where it is possible to identify the
eigenvalues of the generator of $X$.
\begin{theorem}
Considering the function $\psi$ in (\ref{EqPsi}), then we have that
$\lambda$ is an eigenvalue of $L$ if and only if $\psi(b,
\lambda)=0$.
\end{theorem}
\begin{proof}
Since it holds $L\psi+\lambda \psi=0$ and $\psi(a,\lambda)=0$, if we
are told that $\psi(b, \lambda)=0$, then $\psi$ would be an
eigenfunction, concecuently $\lambda$ would be an eigenvalue.

Let us now suppose that $\lambda$ is an eigenvalue of $L$. If that
is the case, then there exists an eigenfunction $\varphi$, thus it
holds that $L \varphi + \lambda \varphi =0$, $\varphi(a,
\lambda)=0$, $\varphi(b, \lambda)=0$. From iii) of Proposition
\ref{PropPsiEq}, we know that $\psi$ also satisfies $L \psi +
\lambda\psi=0$, $\psi(a,\lambda)=0$ and $\psi'(a, \lambda)=1$. And
by Proposition \ref{PropWronskian}, there exists a constant $C$ such
that
$$\varphi(x, \lambda)= C \psi(x, \lambda),$$
which implies that $\psi(b, \lambda)=0$.
\end{proof}

At this point, we are providing grounds to produce a procedure to
generate eigenvalues and eigenfunctions, thus one may be able to
approximate the transition probability densities in (\ref{EqTDF}).

Now, let us produce another stochastic equation which may
additionally help to deal with eigenvalues. The ideas comes from
\cite{McKean0} (see also \cite{Ramirez1}).

For $\lambda$ fixed, consider the Riccati transform
$$P_{t}:=\displaystyle\frac{\psi'(t,\lambda)}{\psi(t,\lambda)},$$
which is valid whenever $\psi$ is not zero.

If we define
$$g(t,x):= e^{x}\cdot
\left[-2\lambda\displaystyle\int_{a}^{t}\psi(z,\lambda)e^{-W(z)}dz
+\displaystyle\frac{1}{e^{W(a)}} \right],$$ then
$\psi'(t,\lambda)=g(t,W(t))$.

Applying the It\^{o}'s formula to the function
$$h(t,x):= \displaystyle\frac{g(t,x)}{\psi(t,\lambda)},$$
and we obtain that $P_{t}$ satisfies the following stochastic
differential equation
\begin{equation}\label{eqnPt}
dP_{t}=\left(-2\lambda-P_{t}^{2}+\frac{1}{2}P_{t}  \right)dt +
P_{t}dW_{t}.
\end{equation}
The relevance of the diffusion $P_{t}$ is given for the following
theorem.

\begin{theorem} \label{valores propios y explosion de Pt}
Consider the diffusion $P_{t}$ started at $+\infty$ at $t=a$ (i.e.
$P_{a+ \epsilon} \rightarrow +\infty$ as $\epsilon \rightarrow
0^{+}$), and restarted at $+\infty$ immediately after any passage to
$-\infty$. Then the number of eigenvalues of $L$ less that $\lambda$
is equal to the number of explosions of $P_{t}$ on $(a,b]$.
\end{theorem}
\begin{proof}
The proof is using the Theorem \ref{teorema funcion psi y sus ceros}
and the Riccati Transform. See e.g. \cite{Mathieu}.
\end{proof}

\begin{remark}
It is remarkable to see that in \cite{Kawazu97} a similar equation
than the one in (\ref{eqnPt}) arrises in the context of the first
passage time of the Brox diffusion. Indeed, such equation describes
the Laplace transform of the first passage time, and they call it
the Kotani's formula.
\end{remark}

\begin{remark}
Our intention is to continue investigating the possible connection
between the eigenvalues of the operator and the minimum values of
the Brownian motion.
\end{remark}

\section{Appendix. Proof of Propostion \ref{CoroZerosEigen}  }

We want to analyze the equation $Lx+\lambda x=0$.
We consider the equation
\begin{equation} \label{anexo 1}
2e^{-W(t)} \left[ Lx(t)+ \lambda x(t)
\right]=(e^{-W(t)}x'(t))'+2\lambda e^{-W(t)}x(t)=0, \; \; a<t<b,
\end{equation}

We are going to use the Pr\"{u}fer method, which we addapt following
the arguments in \cite{Coddington}. In this method, one first
defines $y(t)=e^{-W(t)}x'(t)$. Using (\ref{anexo 1}) we have
\begin{equation} \label{anexo 2}
x'(t)=\displaystyle\frac{y(t)}{e^{-W(t)}}, \; \; \; y'(t)= -2
\lambda e^{-W(t)} x(t).
\end{equation}
Notice that eventhough $y(t)$ is in terms of the BM $W(t)$, the
derivative $y^{\prime}(t)$ is well defined.

The big leap in this classic method is to propose the following
change of coordinates
\begin{equation} \label{anexo 3}
x(t)=r(t)\sin(w(t)), \; \; \; y(t)=r(t)\cos(w(t)).
\end{equation}

Differentiating the equations (\ref{anexo 3}) with respect to $t$ we
have
$$x'(t)=r'(t) \sin(w(t))+r(t)\cos(w(t))w'(t),$$
$$y'(t)=r'(t) \cos(w(t))- r \sin(w(t))w'(t).$$


We now use (\ref{anexo 2}), and solving for $r'$ and $w'$, we obtain
\begin{equation} \label{ecuacion r}
r'(t)= \left(\displaystyle \frac{1}{e^{-W(t)}}- 2\lambda
e^{-W(t)}\right) r(t)\sin(w(t))\cos(w(t)) ,
\end{equation}
and
\begin{equation} \label{ecuacion w}
w'(t)=\displaystyle\frac{1}{e^{-W(t)}}\cos^{2}(w(t))+2 \lambda
e^{-W(t)}\sin^{2}(w(t)).
\end{equation}

These equations with initial conditions have unique solution. If
$\phi$ is a solution of (\ref{anexo 1}) then $\phi$ has the form
$\phi(r)=r(t)\sin(w(t))$. Equation (\ref{ecuacion r}) is of the form
$r^{\prime}(t)=h(t)r(t)$, so the solution $r$ is $r(t)=e^{\int
h(u)du}$, so $r$ is non-negative on $t$.
A consequence of this is that $\phi$ vanishes only when $w$ is a
multiple of $\pi$.

Taking into account the conditions
\begin{equation}\label{condicion 1}
x(a)=0, \ x(b)=0,
\end{equation}
let $\phi(t, \lambda)$ be a nontrivial solution of (\ref{anexo 1}) and (\ref{condicion 1}).

To analyze $\phi$, we now give some properties of $w$, defined in
(\ref{ecuacion w}).

First of all, it holds $w(a,\lambda)=0.$ This is because from
formulas (\ref{anexo 2}) and (\ref{anexo 3}), one has
$$w(a,\lambda)= \tan^{-1}\left(\displaystyle
\frac{\phi(a, \lambda)}{e^{-W(a)}\phi'(a,\lambda)}
 \right)=0.$$

Second, it turns out that for $\lambda$ fixed, $w$ is increasing
function of $t$. To prove this, let us see that, using equation
(\ref{ecuacion w}), the derivative is positive. This is the case if
the eigenvalues are positive, let us check this fact. Let
$(x,\lambda)$ be an eigenpair of the generator, thus
\begin{equation} \label{ecuacion a}
(e^{-W(t)}x'(t))'+2 \lambda e^{-W(t)}x(t)=0.
\end{equation}
Multiplying (\ref{ecuacion a}) by $x(t)$ we obtain
$$x(t)(e^{-W(t)}x'(t))'+2 \lambda e^{-W(t)}x^{2}(t)=0.$$
Integrating and solving for $\lambda$ we arrive at
$$\lambda= \displaystyle\frac{-\int_{a}^{b} x(t)(e^{-W(t)}x'(t))'dt}
{\int_{a}^{b}2 e^{-W(t)}x^{2}(t)dt}.$$ Integrating by parts and
using that $x(a)=x(b)=0$ we have
$$\lambda= \displaystyle\frac{\int_{a}^{b} e^{-W(t)}(x'(t))^{2}dt}
{\int_{a}^{b}2 e^{-W(t)}x^{2}(t)dt} > 0.$$


Now, for fixed $t$, let us see that $w(t,\lambda)$ is monotone
increasing function of $\lambda$. This is actually a consequence of
the following theorem (for a proof see p.210 from
\cite{Coddington}).
\begin{theorem} \label{teorema comparacion}
Let $L_{i}x:=(p_{i}x')'+g_{i}x$. And let $p_{i}$ and $g_{i}$ be
continuous functions on $[a,b]$, such that
$$0<p_{2}(t)\leq p_{1}(t), \; \; \; g_{2}(t)\geq g_{1}(t).$$
Let $L_{1}\phi_{1}=0$ and $L_{2}\phi_{2}=0$ and using the Pr\"{u}fer
method take $w_{1}$ and $w_{2}$ as in (\ref{ecuacion w}) with
$w_{2}(a)\geq w_{1}(a)$. Then
$$w_{2}(t)\geq w_{1}(t), \; \; a \leq t\leq b.$$
\end{theorem}
In our case, $p_{1}(t)=p_{2}(t)=e^{-W(t)}$ and the function $g(t)=2
\lambda e^{-W(t)}$ is increasing in $\lambda$.
















Finally, we have this property of $w$:
\begin{lemma}
$w(b,\lambda)\rightarrow \infty$ as $\lambda \rightarrow \infty$.
\end{lemma}

\begin{proof}
We consider the equation (\ref{anexo 1}). Let $P,G$ be constants
such that for all $t \in [a,b]$
$$e^{-W(t)}\leq P, \; \; \; 2 \lambda e^{-W(t)} \geq G.$$

Now we consider the equation
\begin{equation} \label{ecuacion P G}
Px''(t)+2\lambda G x(t)=0,
\end{equation}
and take $v$ to be the analogous of $w$ with the condition $
v(a,\lambda)=w(a,\lambda)$. Hence, from Theorem \ref{teorema
comparacion} we have that
$$w(t,\lambda)\geq v(t, \lambda).$$

On the other hand, if $f$ is a solution of the equation
(\ref{ecuacion P G}), then we have that for large $\lambda$, $f$ is
of the form
$$f(t)=A \cos \left(\displaystyle  \sqrt{\frac{2 \lambda G}{P}}t \right)+
B \sin \left(\displaystyle  \sqrt{\frac{2 \lambda G}{P}}t \right),$$
where $A$ and $B$ are constants. This solution implies that the
zeros of $f$ increase in number when $\lambda$ is large, because the
periodicity increases. The only way to have that is because $v$ hit
multiples of $\pi$ many times. In particular we have that $v(b,
\lambda) \rightarrow \infty$ as $\lambda\rightarrow \infty$. Since
$w(b,\lambda)\geq v(b, \lambda)$, then we obtain the result.
\end{proof}

Let us see that $w(b,0)\neq\pi$. If that it not the case, the
associted function $\phi(t,0)$ should be an eigenfunction of the
eigenvalue $\lambda=0$. However, as we mentioned before the
eigenvalues are extrictly positive.  Now, since $w(t,0)$ is a
continuous function strictly increasing and $w(a,0)=0$, then we
conclude that $ 0< w(b,0) <\pi$.

On the other hand, since $w(b,\lambda)$ is increasing in $\lambda$
and $w(b,\lambda)\rightarrow \infty$ as $\lambda \rightarrow
\infty$, therefore there exists a first value $\lambda_{1} > 0 $
such that $w(b,\lambda_{1})= \pi.$

Since $w$ is increasing in $t$ we have that
$$0 = w(a,\lambda_{1}) <
w(t,\lambda_{1})< w(b,\lambda_{1}) = \pi,\; \; a<t<b.$$

Then we have that $w(t,\lambda_{1})$ is not a multiple of $\pi$,
hence the solution $\phi(t,\lambda_{1})$ does not vanish in $(a,b)$,
and this function $\phi(t,\lambda_{1})$ is the eigenfunction
associated to the first eigenvalue $\lambda_{1}$.

In the same way, there exist $\lambda_{2}> \lambda_{1}$ such that
$w(b,\lambda_{2})= 2\pi.$ Then the function $\phi(t,\lambda_{2})$ is
the eigenfunction associated with the eigenvalue $\lambda_{2}$ and
has only one zero in $(a,b)$, precisely because $w$ touches only
once the value $\pi$. And the very same reasoning follows to
conclude that the $n$th eigenfunction has exactly $n-1$ zeros in
$(a,b)$.


\end{document}